\documentclass[letterpaper, 10 pt, conference]{ieeeconf}  

\IEEEoverridecommandlockouts                              
\usepackage{graphics}
\usepackage{graphicx} 
\usepackage{rotating}
\usepackage{dcolumn}
\usepackage{longtable}
\usepackage{multirow}
\usepackage{amsmath}
\usepackage{amssymb}
\usepackage{xspace}
\usepackage{textcase}
\usepackage{subcaption}
\usepackage{dblfloatfix}

\usepackage{float}
\usepackage[latin1]{inputenc}
\usepackage{tikz}
\usetikzlibrary{trees}
\usetikzlibrary{arrows}

\DeclareMathOperator*{\argmin}{arg\,min}
\newtheorem{theorem}{Theorem}[section]
\newtheorem{corollary}{Corollary}[theorem]

\renewcommand{\Re}{\mathbb{R}}

\title{\LARGE \bf
Learning Koopman Representations for Hybrid Systems
}

\author{Craig Bakker, Arnab Bhattacharya, Samrat Chatterjee, Casey J. Perkins, and Matthew R. Oster
\thanks{This work was supported by the National Security Directorate Seed LDRD at PNNL}
\thanks{C. Bakker, A. Bhattacharya, S. Chatterjee, C.J. Perkins, and M.R. Oster are with the Pacific Northwest National Laboratory,
        Richland, Washington
        {\tt\small craig.bakker@pnnl.gov}}%
}

\begin{document}

\maketitle
\thispagestyle{empty}
\pagestyle{empty}

\begin{abstract}
The Koopman operator lifts nonlinear dynamical systems into a functional space of observables, where the dynamics are linear.  In this paper, we provide three different Koopman representations for hybrid systems.  The first is specific to switched systems, and the second and third preserve the original hybrid dynamics while eliminating the discrete state variables; the second approach is straightforward, and we provide conditions under which the transformation associated with the third holds.  Eliminating discrete state variables provides computational benefits when using data-driven methods to learn the Koopman operator and its observables. Following this, we use deep learning to implement each representation on two test cases, discuss the challenges associated with those implementations, and propose areas of future work.
\end{abstract}

\section{INTRODUCTION}

For a given dynamical system $\dot{x} = f\left(x\right)$, $x \in \Re^n$, with flow $F^t\left(x\right): \Re^n \rightarrow \Re^n$, the set of Koopman operators (KOs) $\mathcal{K}^t$ forms a continuous semigroup acting on functions $g\left(x\right)$:

\begin{gather}
\left(\mathcal{K}^t g\right) \left(x\right) = g \left(F^t\left(x\right)\right).
\end{gather}

The KO lifts the (nonlinear) state-space dynamics into a functional space of observables, where the dynamics are linear.  A set of functions $S$ is invariant under the KO if $\mathcal{K}^t g \in S \ \forall \ g \in S$.  Invariance is computationally important because if $S$ is finite-dimensional with a specified basis (of observables), the projection of $\mathcal{K}^t$ onto $S$ becomes a matrix $K^t$ that depends on the chosen basis.  In practice, $K^t$ (or $K$, assuming a fixed time step) is typically calculated on an (approximately) invariant functional subspace using time series data.  One approach specifies a dictionary of basis functions (e.g., polynomials, radial basis functions) and solves for $K$ as a linear regression \cite{williams15jsr,korda18jsr}.  Another approach uses neural networks to learn a basis from the data \cite{yeung17jsr}.

In computational studies, it is often valuable to include the state variables as observables so that observable space trajectories map uniquely to state space trajectories.  If the observables are not state-inclusive, the mapping from observables to state variables may be difficult to calculate or non-unique.  State-inclusive observables can impose additional constraints on the Koopman representation in multi-modal systems, but it is still possible to use these observables in such cases \cite{bakker19cp}.

Various KO formulations for dynamical systems with control terms have been developed.  Brunton et al. consider a small system with an analytical finite Koopman representation and produce a Linear Quadratic Regulator feedback control law \cite{brunton16jsr}.  Small changes in the underlying dynamical system, however, produce a control system that is no longer Koopman invariant.  Korda and Mezi\'{c} \cite{korda18jsr} combine the dictionary-of-functions approach with Model Predictive Control (MPC) and a receding time horizon approach.  As shown in Bakker et al. \cite{bakker19jsr}, though, not all Koopman control formulations in the literature are internally consistent (i.e., they may violate the chain rule of calculus).

Almost all of the systems considered in the Koopman literature have continuous dynamics and controls.  That is to say, the state and control variables are continuous, and even the discrete-time implementations are based on continuous trajectories sampled at discrete (and fixed) time intervals.  Budi{\v{s}}i{\'c} et al. \cite{budisic12jsr} allude briefly to the possibility of Koopman dynamics on discrete variables but do not discuss it in detail.  Hanke et al. \cite{hanke18jsr} discretize the control variable in a dynamical system to learn a Koopman representation for each discretized control state; control turns into switching.  The underlying system, however, is fully continuous.  Govindarajan et al. also consider an otherwise continuous system with impulse dynamics \cite{govindarajan16cp}, but there are no discrete state variables in their hybrid pendulum example.

Identifying finite-dimensional sets of Koopman invariant observables is non-trivial when the observables' domains are continuous.  The increased difficulty of finding such sets with a mixture of continuous- and discrete-valued functions has likely been a barrier to the development and use of hybrid Koopman representations.  However, many real-world systems of interest are not fully continuous.  Applying the KO to hybrid systems could provide benefits comparable to those already demonstrated in continuous systems.  

In this paper, we present three different Koopman representations for hybrid systems.  The first is specific to switched systems, and the second and third preserve the original hybrid dynamics while eliminating the discrete state variables; the second approach is straightforward, and we provide conditions under which the transformation associated with the third holds.  Eliminating discrete state variables provides computational benefits when using data-driven methods to learn the Koopman operator and its observables.  Following this, we leverage deep learning to produce data-driven implementations of the representations on two test problems and computationally demonstrate their performance.  These demonstrations show that each Koopman representation can capture the dynamics of the underlying hybrid system.  Based on the computational demonstrations, we discuss the challenges associated with learning hybrid Koopman representations and suggest avenues for future research in this area.

\section{Koopman Representations of Hybrid Dynamical Systems}
\label{Koopman Representations of Hybrid Dynamical Systems}

A general discrete-time hybrid system with state variables $x$ and $y$ and control variables $u$ and $z$ is

\begin{gather}
x_{k+1} = f\left(x_k,y_k,u_k,z_k\right)  
\label{hybrid 1} \\
y_{k+1} = g \left(x_k,y_k,u_k,z_k\right),
\label{hybrid 2}
\end{gather}

\noindent where $x \in \Re^n$ and $u \in \Re^m$ are continuous, while  $y \in \mathbb{Z}^p$ and  $z \in \mathbb{Z}^q$ are discrete.  We consider the discrete-time case because most data-driven approaches rely on discretely sampled trajectory data.  For the proposed Koopman representations of these systems, we will use variations on the form

\begin{gather}
\psi_x \left(x_{k+1}\right) = K_x \psi_x \left(x_k\right) + K_{xu} \psi_{xu} \left(x_k,u_k\right),
\end{gather}

\noindent where $K_x$ and $K_{xu}$ are the finite KO approximations and $\psi_x\left(x\right)$ and $\psi_{xu} \left(x,u\right)$ are KO observables.  For more on the use of this form, see Bakker et al. \cite{bakker19jsr}.  Uncontrolled systems then become a special case where $K_{xu} = 0$.  In principle, we could extend this to

\begin{gather}
\psi_x \left(x_{k+1},y_{k+1}\right) = K_x \psi_x \left(x_k,y_k\right) \nonumber \\
+ K_{xu} \psi_{xu} \left(x_k,y_k,u_k,z_k\right)
\end{gather}

\noindent for a hybrid system and treat $y$ no differently than $x$. However, such a system may have difficulty in accounting for the discrete nature of $y$ -- especially if $\psi_x$ is state-inclusive.  On the one hand, identifying an (approximately) invariant functional subspace with continuous and discrete domains may be very difficult.  On the other hand, using a continuous relaxation of $y$ to identify its correct discrete value would actually change the dynamics in (\ref{hybrid 1})-(\ref{hybrid 2}).  Additional structure is necessary to ensure that the discrete-valued $y_{k+1}$ can always be recovered uniquely from $\psi_x \left(x_{k+1},y_{k+1}\right)$.  We therefore propose three different KO representations.

\subsection{Switched Systems}
\label{Switched Systems}

Switched systems are a special case of hybrid systems with the switched states' dynamics indexed by $\lambda \in \mathbb{N}$:

\begin{gather}
x_{k+1} = f_{\lambda_k} \left(x_k,u_k\right) \\
\lambda_{k+1} = g \left(x_k,u_k,\lambda_k\right).
\end{gather}

Their structure allows for a simpler KO representation than hybrid systems in general:

\begin{gather}
\psi_x \left(x_{k+1}\right) = K^{\lambda_k}_x \psi_x \left(x_k\right) + K^{\lambda_k}_{xu} \psi_{xu} \left(x_k,u_k\right) \\
\lambda_{k+1} = g \left(x_k,u_k,\lambda_k\right).
\end{gather}

$\lambda$ now indexes a set of $K_x$ and $K_{xu}$ matrices (i.e., distinct $K^{\lambda}_x$ and $K^{\lambda}_{xu}$ matrices for each value of $\lambda$).  Different switched states use the same set of observables instead of having different sets of observables for each state, and there is good reason for this.  The learned KO observables represent a basis for an (approximately) invariant KO subspace $S$; the finite representation of the KO on $S$ (i.e., $K_x$) depends on $S$ and maps $S$ to itself.  A different set of observables would span a different invariant subspace $S'$ and therefore have a different KO representation.  Trying to use the KO to map from one set to the other could therefore not use either finite representation consistently.  There may be ways to circumvent this problem, but for the purpose of this paper, and for the sake of simplicity, we will not attempt to do so here.

\subsection{General Hybrid Systems}
\label{General Hybrid Systems}

Here, we propose using a set of continuous observables $\xi \left(x,y\right)$ and $\zeta\left(u,z\right)$ that map uniquely to $y$ and $z$, respectively.  For a state-inclusive controlled system, we have

\begin{gather}
\psi_x \left(x_{k+1},y_{k+1}\right) = K_x \psi_x \left(x_k,y_k\right) \nonumber \\
+ K_{xu} \psi_{xu} \left(x_k,y_k,u_k,z_k\right) \\
\psi_x \left(x,y\right) = \left[ x \ \xi\left(x,y\right) \ \varphi_x\left(x,y\right) \right] \\
\psi_{xu} \left(x,y,u,z\right) = \left[ u \ \zeta\left(u,z\right) \ \varphi_{xu} \left(x,y,u,z\right) \right],
\end{gather}

\noindent where $\varphi_x\left(x,y\right)$ and $\varphi_{xu} \left(x,y,u,z\right)$ are observable functions (either fixed or learned).  Functions such as

\begin{gather}
\xi \left(x,y\right) = \left(2 y - 1\right) \log \left( 1 + e^{\rho \left(x,y\right)}\right) 
\label{softmax formulation} \\
\xi \left(x,y\right) = \frac{1}{2} \tanh \left( \rho \left(x,y\right) \right) + y
\label{tanh formulation}
\end{gather}

\noindent for binary and integer variables, respectively, provide a surjective mapping from $\xi\left(x,y\right)$ to $y$; $\rho\left(x,y\right)$ is a function that can be fixed or learned just like $\varphi_x \left(x,y\right)$.  (\ref{softmax formulation}) and (\ref{tanh formulation}) are examples of possible options; others could be defined.  Analogous definitions apply to $\zeta\left(u,z\right)$.  For both $\xi \left(x,y\right)$ and $\zeta \left(u,z\right)$, the surjective mapping makes it possible to eliminate the discrete variables without affecting the underlying dynamics.  This in turn facilitates the use of data-driven computational methods to learn the KO and its observables.

\subsection{Transformed Systems}
\label{Transformed Systems}

Here, we introduce a different transformation between (\ref{hybrid 1})-(\ref{hybrid 2}) and an equivalent system with continuous state variables; again, we retain the hybrid dynamics using only continuous variables.  We then provide conditions for mapping the hybrid and continuous sets of trajectories to each other.

\begin{theorem}
\label{transf ctrl}

Consider the system defined in (\ref{hybrid 1})-(\ref{hybrid 2}), and define $\sigma = \left[ y \ z \right]$.  Fixed control inputs $\left\{u_k,z_k\right\}$ and initial conditions $\left(x_0,y_0\right)$ produces a trajectory $\left\{x_k,y_k\right\}$, $k \in \mathbb{N}$.  If

\begin{gather}
f\left(x,y_1,u,z_1\right) \neq f\left(x,y_2,u,z_2\right)
\label{unique transf 2}
\end{gather}

\noindent for all $x$ and $u$, and for all $\sigma_1 \neq \sigma_2$, then there exists a dynamical system with continuous state variables $s \in \Re^{2n+m}$ and corresponding trajectory $\left\{s_k\right\}$ such that there exists a bijection between $s_k$ and $\left(x_k,y_k\right)$, $k \in \mathbb{N}$.

\end{theorem}

\begin{proof}

Let $S \subseteq \mathbb{Z}^{p+q}$ be the set of possible $\sigma$ values, and define $v$, $\phi$, and $\omega$ such that

\begin{gather}
v_k = x_{k+1} - x_k 
\label{x ctrl eqn} \\
v_k = f\left(x_k,y_k,u_k,z_k\right) - x_k \\
u_{k+1} = u_k + \phi_k 
\label{u ctrl eqn} \\
z_{k+1} = z_k + \omega_k \\
v_{k+1} = f\left(x_{k+1},y_{k+1},u_{k+1},z_{k+1}\right) - x_{k+1} \nonumber \\
= f\left(x_k + v_k,g \left(x_k,y_k,u_k,z_k\right),u_k+\phi_k,z_k+\omega_k\right) \nonumber \\
- \left(x_k + v_k\right).
\end{gather}

If (\ref{unique transf 2}) holds, then there exists a family of bijections $G$: 

\begin{gather}
\tilde{f}\left(x,u,\sigma_1\right) \neq \tilde{f}\left(x,u,\sigma_2\right) \ \forall \ x,u; \ \forall \ \sigma_1 \neq \sigma_2 \\
V\left(x,u\right) = \left\{ v: v = \tilde{f}\left(x,u,\sigma\right) - x \ \forall \ \sigma \in S \right\} \\
G\left(\cdot,x,u\right) : V\left(x,u\right) \rightarrow S \\
G\left(v,x,u\right) = \sigma : v = \tilde{f}\left(x,u,\sigma\right) - x \\
G\left(v_k,x_k,u_k\right) = \sigma_k \\
G_y\left(v_k,x_k,u_k\right) = y_k \\
G_z\left(v_k,x_k,u_k\right) = z_k \\
y_{k+1} = g \left(x_k,G_y\left(v_k,x_k,u_k\right),u_k,G_z\left(v_k,x_k,u_k\right)\right) \nonumber \\
= h\left(x_k,u_k,v_k\right) \\
v_{k+1} = f\left(x_k + v_k, h\left(x_k,u_k,v_k\right),u_k + \phi_k, \right. \nonumber \\
\left. G_z\left(v_k,x_k,u_k\right) + \omega_k \right) - \left(x_k + v_k\right).
\label{v ctrl eqn}
\end{gather}

The dynamics then consist of (\ref{x ctrl eqn}), (\ref{u ctrl eqn}), and (\ref{v ctrl eqn}) with initial conditions defined by $x_0$, $y_0$, $u_0$, $u_1$, $z_0$, and $z_1$:

\begin{gather}
v_0 = f\left(x_0,y_0,u_0,z_0\right) - x_0 \\
\phi_0 = u_1 - u_0 \\
\omega_0 = z_1 - z_0.
\end{gather}

Given $u_k$ and $z_k$, we have a unique and bijective mapping between $\left(x_k,y_k\right)$ and $s_k = \left(x_k,u_k,v_k\right)$.  The new dynamical system has continuous state variables $x$, $v$, and $u$ as well as control variables $\phi$ (continuous) and $\omega$ (discrete).

\end{proof}

The proof of Theorem \ref{transf ctrl} uses the definition $v_k = x_{k+1} - x_k$.  In this proof, $G\left(\cdot,x,u\right)$ maps from $V\left(x,u\right) \subset \Re^n$, not $\Re^n$, to $S$.  However, it is possible to extend this.  For example, using a mapping $\mathcal{V} \left(\cdot,x,u\right): \Re^n \rightarrow V\left(x,u\right)$ such as

\begin{gather}
\mathcal{V} \left(v,x,u\right) = \argmin \limits_{v' \in V\left(x,u\right)} \left\| v' - v\right\|
\end{gather}

\noindent produces the surjective mapping $G\left(\mathcal{V} \left(\cdot,x,u\right),x,u\right): \Re^n \rightarrow S$; the mapping from $s_k$ to $\left(x_k,y_k\right)$ then becomes surjective.  This extension is computationally useful because it makes the original hybrid trajectory recoverable from the transformed trajectory even in the presence of noise or other errors.

\begin{corollary}
\label{transf unctrl}

Consider the uncontrolled hybrid system

\begin{gather}
x_{k+1} =  f \left(x_k,y_k\right)
\label{x transf unctrl} \\
y_{k+1} = g \left(x_k,y_k\right),
\label{y transf unctrl}
\end{gather}

\noindent where $x \in \Re^n$ and $y \in \mathbb{Z}^p$. Fixed initial conditions $\left(x_0,y_0\right)$ produce a trajectory $\left\{x_k,y_k\right\}$, $k \in \mathbb{N}$.  If

\begin{gather}
f\left(x,y_1\right) \neq f\left(x,y_2\right)
\label{unique transf 1}
\end{gather}

\noindent for all $x$, and for all $y_1 \neq y_2$, then there exists a dynamical system with continuous state variables $s \in \Re^{2n}$ and corresponding trajectory $\left\{s_k\right\}$ such that there exists a bijection between $s_k$ and $\left(x_k,y_k\right)$, $k\in \mathbb{N}$.

\end{corollary}

\begin{proof}
The result follows from Theorem \ref{transf ctrl} if $u_k$ and $z_k$ are set identically to zero.
\end{proof}

To generate the necessary time series data, $v_k$ could be calculated after the simulations were run, and learning the KO would implicitly learn the mapping from $\left(x_k,y_k\right)$ to $v_k$.  If (\ref{unique transf 1}) does not hold, though, then $y$ is not identifiable from $v$.  Defining additional variables: $w_{k+1} = x_k$ and $u_{k+1} = v_k$ could circumvent this problem.  For an uncontrolled system, the necessary condition would be

\begin{gather}
f\left(f\left(x,y_1\right),g\left(x,y_1\right)\right) \neq f\left(f\left(x,y_2\right),g\left(x,y_2\right)\right) \nonumber \\
\text{or } f\left(x,y_1\right) \neq f\left(x,y_2\right)  \ \forall \ x,y_1 \neq y_2.
\label{multi-step cond}
\end{gather}

(\ref{unique transf 1}) defines a single-step identifiability criterion. (\ref{multi-step cond}) defines a two-step version of that criterion: if the trajectories corresponding to different $y$ values differ after two time steps, then $y$ is identifiable by considering two successive data points from the same trajectory.  In principle, eventually an $n$-step criterion (with associated additional variables) could be found, and there would be an analogous bijective mapping to the discrete variables.  Additionally, the transformation described in the proof of Theorem \ref{transf ctrl} simplifies the functional forms of $K_x$ and $K_{xu}$, since the $x$ dynamics are known to be linear in $x$ and $v$ and the $u$ dynamics are linear in $u$ and $\phi$.  The resulting transformed KO representation is

\begin{gather}
\psi_x \left(x_{k+1},u_{k+1},v_{k+1}\right) = K_x \psi_x \left(x_k,u_k,v_k\right) \nonumber \\
+ K_{xu} \psi_{xu} \left(x_k,u_k,v_k,\phi_k,\omega_k\right) \\
\psi_x \left(x,u,v\right)= \left[x \ u \ v \ \varphi_x\left(x,u,v\right)\right] \\
\psi_{xu} \left(x,u,v,\phi,\omega\right) = \left[ \phi_k \ \omega_k \ \varphi_{xu}\left(x,u,v,\phi,\omega\right) \right].
\end{gather}

\begin{gather}
K_x = \left[ \begin{array}{cccc} I_n &\mathbf{0}_{n \times m} &I_n &\mathbf{0}_{n \times n_{\varphi_x}} \\ \mathbf{0}_{m \times n} &I_m &\mathbf{0}_{m \times n} &\mathbf{0}_{m \times n_{\varphi_x}} \\ \multicolumn{4}{c}{\tilde{K}_x}   \end{array} \right] \\
K_{xu} = \left[ \begin{array}{ccc} \mathbf{0}_{n \times m} &\mathbf{0}_{n \times q} &\mathbf{0}_{n \times n_{\varphi_{xu}}} \\ I_m &\mathbf{0}_{m\times q} &\mathbf{0}_{m \times n_{\varphi_{xu}}} \\ \multicolumn{3}{c}{\tilde{K}_{xu}} \end{array} \right].
\end{gather}

The uncontrolled case has analogous structure.  Using this transformation transfers the discontinuities from the state space into the dynamics, as does the general hybrid formulation described previously.  In both cases, the motivation behind the transfer is computational.  For learning the KO and its observables from data, and for using the KO to simulate trajectories, implicit hybrid or switching behaviour poses fewer difficulties than handling discrete state variables directly.  This is particularly true for when using neural networks to learn the KO observables.  Neural networks that produce discrete outputs cannot be trained using gradient-based learning algorithms because the outputs are effectively piecewise constant (i.e., their gradients are zero).

\section{Computational Demonstrations}
\label{Computational Demonstrations}

\subsection{Implementation}
\label{Implementation}

For testing, we used two computational examples: a simple numerical example and a vehicle automatic transmission.  The neural networks used to learn the Koopman observables were implemented in Tensorflow \cite{tensorflow}.  These networks used Exponential Linear Unit (ELU) hidden layer activations, were trained to minimize relative Mean Squared Error (MSE)

\begin{gather}
\sum_k \frac{\left\| \psi \left(x_{k+1}\right) - K \psi \left(x_k\right) \right\|^2}{\left\| x_{k+1} - x_k \right\|^2},
\label{MSE error}
\end{gather}

\noindent and were trained using Adam \cite{kingma2014adam} with an initial learning rate of 0.001.  For the numerical test case, we used networks with four layers, each eight neurons wide (8x4) to represent $\varphi_x$, $\varphi_{xu}$ and $\rho$.  The automatic transmission model, being more complicated, used a 12x4 network for each of $\varphi_x$, $\varphi_{xu}$ and $\rho$; $\varphi_x$ and $\varphi_{xu}$ each had five elements for the numerical test case and ten elements for the automatic transmission.  The training data was generated from grids of initial points on the state and control variable spaces run for a single time step: a total of 1000 points for the numerical test case and 625,000 points for the automatic transmission model.  The modified control variables in the transformed formulation changed the dimension of the space being sampled, so different sampling grids were used.  This grid-based approach sampled the space uniformly without requiring ergodicity.  We defined the grids such that the system dynamics would lie within the grid interior, which avoided boundary effects or a need for extrapolation.  In the numerical example, mini-batch training sufficed.  However, for the automatic transmission, full batch training was needed for good training convergence.

To generate trajectories from learned KO representations, we tested two different implementations.  The first implementation (designated as `multiplied') initialized $\psi_x$ and propagated the dynamics by multiplying by $K_x$ and adding $K_{xu} \psi_{xu}$ (if necessary).  State variable values were extracted in post-processing, but other than the initialization and calculation of $\psi_{xu}$, $x_k$ was not explicitly needed to calculate the trajectory.  Control terms aside, this produced a purely linear representation of the dynamics.  The second implementation (designed as `evaluated') initialized and multiplied $\psi_x$ as well, but at each iteration, it extracted state variable values and re-initialized $\psi_x$.  This implementation meant that the dynamics were not only a sequence of matrix multiplications, but it was intended to reduce error propagation across iterations and align with a rolling time horizon approach (such as MPC).

With the general hybrid formulation, we obtained better results when we learned a relaxed representation ($\xi_i = y_i$) to a lower level of fidelity and used that as a warm start to learn the representation using (\ref{tanh formulation}) as opposed to learning $\rho$ and $\varphi_x$ simultaneously.  For the automatic transmission, there was also the question of whether to learn dynamics and control sequentially or whether to try and learn both simultaneously.  Sequential learning was sometimes beneficial in speeding up convergence, but generally, learning both simultaneously provided the greatest accuracy (with the warm start approximation as described for the general hybrid formulation).

\subsection{Numerical Test Case}
\label{Numerical Test Case}

As an initial test case, we defined the hybrid system

\begin{gather}
x_{k+1} = x_k + \frac{2 y_k - 1}{1+0.1x_k^2} \\
y_{k+1} = \left\{\begin{array}{cc} 1 &x_k < -1 \\ 0 &x_k > 1 \\ y_k & \text{else} \end{array} \right. .
\end{gather}

\begin{figure}[htp]
\centering
\begin{subfigure}[t]{0.23\textwidth}
\centering
\includegraphics[width=\textwidth]{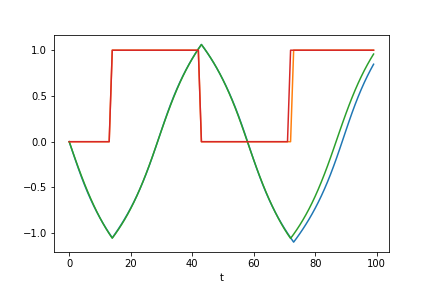}
\caption{Switched, Eval, 0.001}
\end{subfigure}
\begin{subfigure}[t]{0.23\textwidth}
\centering
\includegraphics[width=\textwidth]{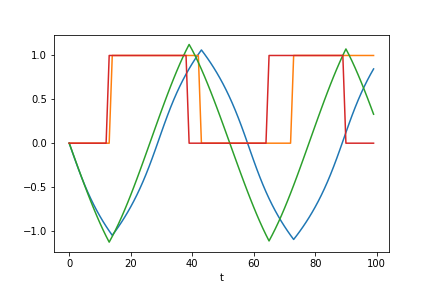}
\caption{Switched, Mult, 0.001}
\end{subfigure}
\begin{subfigure}[t]{0.23\textwidth}
\centering
\includegraphics[width=\textwidth]{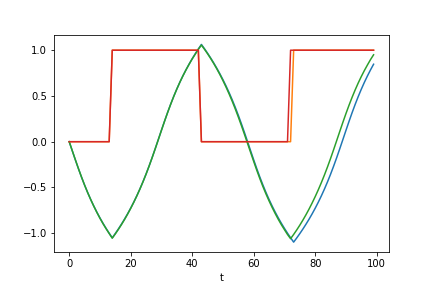}
\caption{Hybrid, Eval, 0.002}
\end{subfigure}
\begin{subfigure}[t]{0.23\textwidth}
\centering
\includegraphics[width=\textwidth]{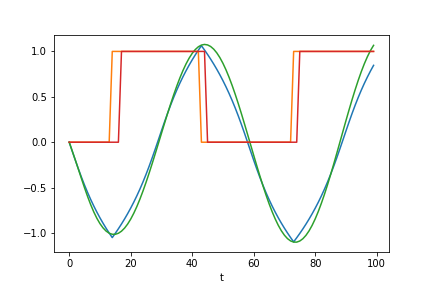} 
\caption{Hybrid, Mult, 0.002}
\end{subfigure}
\begin{subfigure}[t]{0.23\textwidth}
\centering
\includegraphics[width=\textwidth]{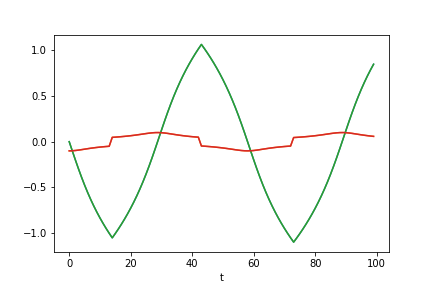}
\caption{Transformed, Eval, 0.005}
\end{subfigure}
\begin{subfigure}[t]{0.23\textwidth}
\centering
\includegraphics[width=\textwidth]{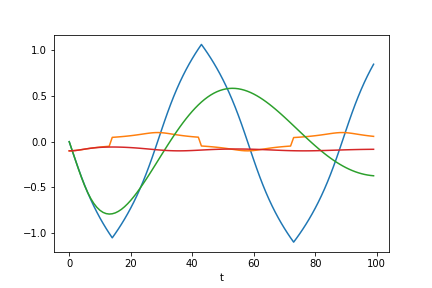}
\caption{Transformed, Mult, 0.005}
\end{subfigure}
\caption{Simulation Results (Formulation, Method, MSE; True $x$ and $y$ or $v$ Trajectories Shown in Blue and Orange, Respectively; Koopman-Generated $x$ and $y$ or $v$ Trajectories Shown in Green and Red, Respectively)}
\label{fig:test case results}
\end{figure}

This use case satisfies the uniqueness criterion in (\ref{unique transf 1}).  Fig. \ref{fig:test case results} shows sample trajectories produced by different learned Koopman representations; each caption's MSE value, defined by (\ref{MSE error}), indicates the training error achieved by each method.  These trajectory plots assess  multi-step testing error; note the piecewise continuous nature of the true $v$ trajectory.  The switched representation was the easiest to train of the three methods, and it was also able to get the lowest training error; the transformed formulation was able to train to an error of 0.005, and the general hybrid formulation using (\ref{tanh formulation}) was able to train to 0.002.  In all three cases, the evaluated method produced more accurate Koopman trajectories -- trajectories that remained accurate over the 100 time steps presented in Fig. \ref{fig:test case results}.  The multiplied method still produced good results in the switched and general hybrid formulations, and the worse results for the transformed formulation may be a reflection of its higher training error.

\subsection{Vehicle Automatic Transmission}
\label{Vehicle Automatic Transmission}

\begin{figure*}[htp]
\centering
\begin{subfigure}[t]{0.3\textwidth}
\centering
\includegraphics[width=\textwidth]{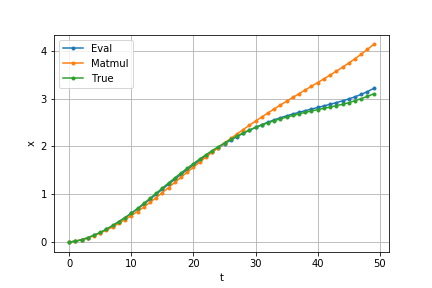} 
\caption{Position, Switched, 0.005}
\end{subfigure}
\begin{subfigure}[t]{0.3\textwidth}
\centering
\includegraphics[width=\textwidth]{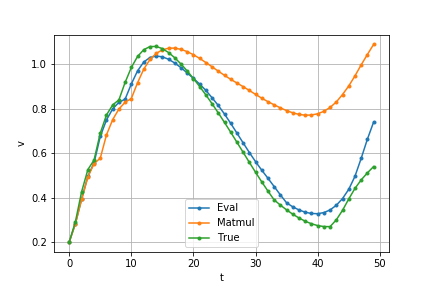}
\caption{Velocity, Switched, 0.005}
\end{subfigure}
\begin{subfigure}[t]{0.3\textwidth}
\centering
\includegraphics[width=\textwidth]{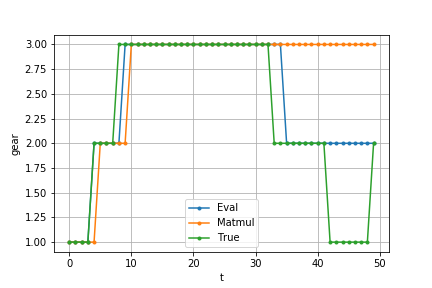}
\caption{Gear, Switched, 0.005}
\end{subfigure}
\begin{subfigure}[t]{0.3\textwidth}
\centering
\includegraphics[width=\textwidth]{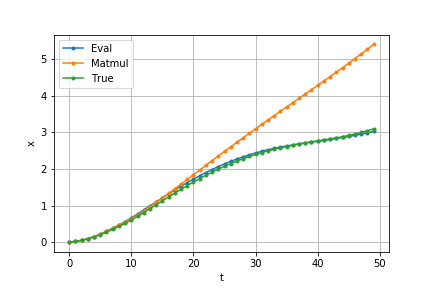} 
\caption{Position, Hybrid, 0.01}
\end{subfigure}
\begin{subfigure}[t]{0.3\textwidth}
\centering
\includegraphics[width=\textwidth]{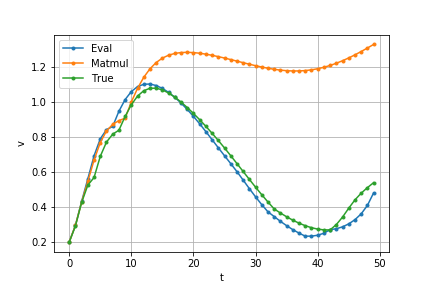}
\caption{Velocity, Hybrid, 0.01}
\end{subfigure}
\begin{subfigure}[t]{0.3\textwidth}
\centering
\includegraphics[width=\textwidth]{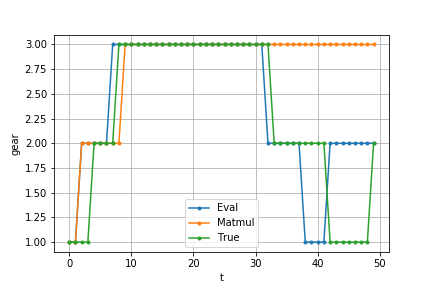}
\caption{Gear, Hybrid, 0.01}
\end{subfigure}
\begin{subfigure}[t]{0.3\textwidth}
\centering
\includegraphics[width=\textwidth]{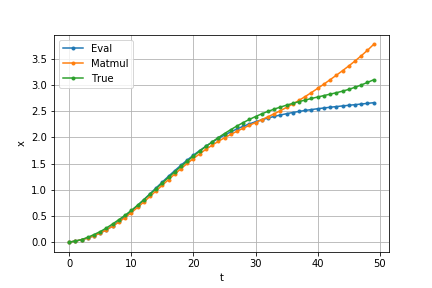} 
\caption{Position, Transformed, 0.008}
\end{subfigure}
\begin{subfigure}[t]{0.3\textwidth}
\centering
\includegraphics[width=\textwidth]{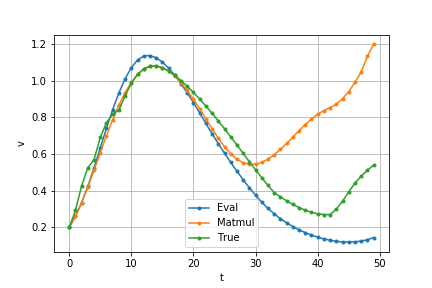}
\caption{Velocity, Transformed, 0.008}
\end{subfigure}
\begin{subfigure}[t]{0.3\textwidth}
\centering
\includegraphics[width=\textwidth]{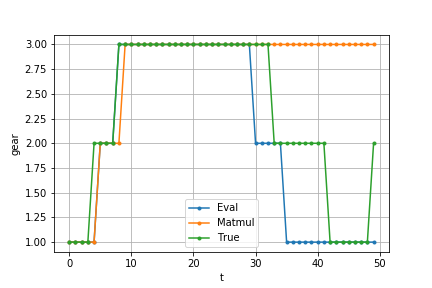}
\caption{Gear, Transformed, 0.008}
\end{subfigure}
\caption{Automatic Transmission Simulation Results (Variable, Method, MSE)}
\label{fig:transmission results}
\end{figure*}

We used an automatic transmission model similar to those of Antsaklis and Koutsoukos \cite{antsaklis02jsr} and of Lygeros et al. \cite{lygeros08tr}: 

\begin{gather}
x_{k+1} = x_k + v_k \Delta t \\
v_{k+1} = v_k + \left[-\alpha v_k^2 \text{sign} \left(v_k\right) - g \sin \left(\theta \left(x_k\right)\right) \right. \\
\left. - \gamma v_k G \left(q_k\right) + T_k \eta \left(v_k G \left(q_k\right) \right) \right] \Delta t \\
q_{k+1} = \left\{ \begin{array}{cc} 
q_k + 1 & v_k G \left(q_k\right) > \omega_{high}, q_k < 5 \\
q_k - 1 & v_k G \left(q_k\right) < \omega_{low}, q_k < 1 \\
q_k		  & \text{else} \end{array} \right. .
\end{gather}

Here, $-\alpha v_k^2 \text{sign} \left(v_k\right)$ is deceleration due to air resistance, $-g \sin \left(\theta \left(x_k\right)\right)$ is acceleration due to gravity on a sloping road, $- \gamma v_k G \left(q_k\right)$ is deceleration due to engine friction, and $T \eta \left(v_k G \left(q_k\right) \right)$ is the acceleration due to torque provided by the engine; $x_k$ is position, $v_k$ is velocity, $q_k$ is the gear (1 to 5), and $G\left(q\right)$ is the gear ratio for each gear of the transmission.  Fig. \ref{fig:transmission results} shows trajectory results for the different formulations with a fixed $\left\{T_k\right\}$ series and defined functions $\theta \left(x\right)$ and $G\left(q\right)$; again, the plots provide an assessment of multi-step testing error, and the listed MSE values indicate training error as defined by (\ref{MSE error}).  The switched formulation trained to MSE = 0.005 while the general hybrid formulation trained to 0.01 and the transformed formulation trained to 0.008.  Again, the evaluated trajectories (blue) followed the true trajectory (green) more closely than the multiplied trajectories (orange).  The evaluated position and velocity trajectories remained at least qualitatively accurate through for 40-50 time steps, whereas the multiplied trajectories were only good for 20-30 time steps.  The discrepancies seemed worst when the Koopman trajectories produced $q_k$ errors.

\section{Discussion}

Overall, the general hybrid formulation was the hardest to train.  The switched formulation was the easiest to train and the most accurate, and the transformed formulation trained quickly but tended to plateau short of the desired accuracy.  Larger neural networks might provide greater accuracy but would also require more effort to train.  The general hybrid formulation may be the most scalable.  As the number of discrete variables increases, the number of distinct states (and thus the number of different KO matrices to learn) for the switched formulation may grow exponentially.  For the transformed formulation, recovering the discrete variable values and guaranteeing that (\ref{unique transf 2}) or (\ref{unique transf 1}) holds may become more difficult as $p$ increases; this formulation's additional variables also increase the sampling burden.

The neural networks used to learn the Koopman observables were small but required a relatively large amount of data to provide a sufficient level of accuracy.  Mini-batch training can be a way to handle those large datasets, but training stability can suffer as a result.  Strategic approaches may be useful here.  For example, we found it advantageous to do initial training with smaller sets of data and then to use the larger, more time-consuming datasets after that pre-training had completed; this was like mesh refinement in finite element methods.  Similarly, for the general hybrid formulation, doing pre-training with a simple continuous relaxation of the discrete variables $y$ made it easier to get final convergence.  These observations point towards larger questions of scalability and computational complexity for the methods presented.  The current paper does not seek to address those questions in depth, but future work should consider those issues in the interest of making the methods more widely applicable.

Error propagation is a key factor in producing useful Koopman representations.  Once training converged to the desired mean error level, the single-step error distributions were generally long-tailed.  What may be even more important, though, are errors in the discrete transitions: significant trajectory deviations usually began with an inaccurate discrete variable transition from one state to another (or lack of such transition).  Accurately predicting these transitions is of primary importance for long-term accuracy.  One solution would be to change the training function.  Weighting the training data so that $\left(x_k,x_{k+1}\right)$ pairs representing discrete state transitions are given more importance would be one way to do this.  Multi-step training would be another approach: it would replace (\ref{MSE error}) with something like

\begin{gather}
\min \sum_k \frac{\left\| \psi_x \left(x_{k+2}\right) - K^2_x \psi_x \left(x_k\right)\right\|^2}{\left\| x_{k+2} - x_k \right\|^2}.
\end{gather}

This would help to enforce accuracy across discrete transitions (e.g., at step $k+1$).  MSE may also not be the ideal metric for these problems.  In the general hybrid formulation, for example, small MSE can lead to large trajectory errors because $y$ is off by 1 even though the error on $\xi$ is small.

\section{Conclusions}

We have produced and described three Koopman representations for hybrid systems.  Two representations rely on surjective mappings to eliminate discrete variables.  Initial demonstrations have provided some insight into how these methods could be scaled up to larger systems and applied in other contexts.  One particular area for future work will be in modifying the training process to account for the specific challenges of hybrid systems; customized error functions (other than MSE) and multi-step training would fall under this heading.  Another area will be comparing different discrete-to-continuous mapping functions for the general hybrid formulation in terms of trainability and accuracy.  Finally, it may be possible to produce additional hybrid Koopman formulations and compare them with the ones presented here.

\section{ACKNOWLEDGMENTS}

The research described in this paper was funded by the National Security Directorate Seed Laboratory Directed Research and Development Investment at the Pacific Northwest National Laboratory, a multiprogram national laboratory operated by Battelle for the U.S. Department of Energy.


\bibliography{bib}

\begin{thebibliography}{10}

\bibitem{williams15jsr}
M.~O. Williams, I.~G. Kevrekidis, and C.~W. Rowley, ``A data--driven
  approximation of the koopman operator: Extending dynamic mode
  decomposition,'' {\em Journal of Nonlinear Science}, vol.~25, no.~6,
  pp.~1307--1346, 2015.

\bibitem{korda18jsr}
M.~Korda and I.~Mezi{\'c}, ``Linear predictors for nonlinear dynamical systems:
  Koopman operator meets model predictive control,'' {\em Automatica}, vol.~93,
  pp.~149--160, 2018.

\bibitem{yeung17jsr}
E.~Yeung, S.~Kundu, and N.~Hodas, ``Learning deep neural network
  representations for koopman operators of nonlinear dynamical systems,'' {\em
  arXiv preprint arXiv:1708.06850}, 2017.

\bibitem{bakker19cp}
C.~Bakker, K.~E. Nowak, and W.~S. Rosenthal, ``Learning koopman operators for
  systems with isolated critical points,'' in {\em 2019 IEEE 58th Conference on
  Decision and Control (CDC)}, IEEE, 2019.

\bibitem{brunton16jsr}
S.~L. Brunton, B.~W. Brunton, J.~L. Proctor, and J.~N. Kutz, ``Koopman
  invariant subspaces and finite linear representations of nonlinear dynamical
  systems for control,'' {\em PloS one}, vol.~11, no.~2, p.~e0150171, 2016.

\bibitem{bakker19jsr}
C.~Bakker, W.~S. Rosenthal, and K.~E. Nowak, ``Koopman representations of
  dynamic systems with control,'' {\em arXiv preprint arXiv:1908.02233}, 2019.

\bibitem{budisic12jsr}
M.~Budi{\v{s}}i{\'c}, R.~Mohr, and I.~Mezi{\'c}, ``Applied koopmanism,'' {\em
  Chaos: An Interdisciplinary Journal of Nonlinear Science}, vol.~22, no.~4,
  p.~047510, 2012.

\bibitem{hanke18jsr}
S.~Hanke, S.~Peitz, O.~Wallscheid, S.~Klus, J.~B{\"o}cker, and M.~Dellnitz,
  ``Koopman operator based finite-set model predictive control for electrical
  drives,'' {\em arXiv preprint arXiv:1804.00854}, 2018.

\bibitem{govindarajan16cp}
N.~Govindarajan, H.~Arbabi, L.~van Blargian, T.~Matchen, E.~Tegling, and
  I.~Mezi{\'c}, ``An operator-theoretic viewpoint to non-smooth dynamical
  systems: Koopman analysis of a hybrid pendulum,'' in {\em 2016 IEEE 55th
  Conference on Decision and Control (CDC)}, pp.~6477--6484, IEEE, 2016.

\bibitem{tensorflow}
M.~Abadi, A.~Agarwal, P.~Barham, E.~Brevdo, Z.~Chen, C.~Citro, G.~S. Corrado,
  A.~Davis, J.~Dean, M.~Devin, S.~Ghemawat, I.~Goodfellow, A.~Harp, G.~Irving,
  M.~Isard, Y.~Jia, R.~Jozefowicz, L.~Kaiser, M.~Kudlur, J.~Levenberg,
  D.~Man\'{e}, R.~Monga, S.~Moore, D.~Murray, C.~Olah, M.~Schuster, J.~Shlens,
  B.~Steiner, I.~Sutskever, K.~Talwar, P.~Tucker, V.~Vanhoucke, V.~Vasudevan,
  F.~Vi\'{e}gas, O.~Vinyals, P.~Warden, M.~Wattenberg, M.~Wicke, Y.~Yu, and
  X.~Zheng, ``{TensorFlow}: Large-scale machine learning on heterogeneous
  systems,'' 2015.
\newblock Software available from tensorflow.org.

\bibitem{kingma2014adam}
D.~Kingma and J.~Ba, ``Adam: A method for stochastic optimization,'' {\em arXiv
  preprint arXiv:1412.6980}, 2014.

\bibitem{antsaklis02jsr}
P.~J. Antsaklis, X.~D. Koutsoukos, and N.~Dame, ``Hybrid systems control,''
  {\em Encyclopedia of Physical Science and Technology}, vol.~7, pp.~445--458,
  2002.

\bibitem{lygeros08tr}
J.~Lygeros, C.~Tomlin, and S.~Sastry, ``Hybrid systems: Modeling, analysis and
  control,'' tech. rep., University of California, Berkeley, 2008.

\end{thebibliography}
\bibliographystyle{ieeetr}

\end{document}